\documentclass{amsart}

\usepackage{amsmath,amssymb,amsthm}
\usepackage{a4wide}
\usepackage{hyperref}

\usepackage{graphicx}
\usepackage{xypic}
\entrymodifiers={+!!<0pt,\fontdimen22\textfont2>}
\usepackage[all]{xy}

\newtheoremstyle{myremark} % name
    {7pt}                    % Space above
    {7pt}                    % Space below
    {}  	                 % Body font
    {}                           % Indent amount
    {\bf}       	         % Theorem head font
    {.}                          % Punctuation after theorem head
    {.5em}                       % Space after theorem head
    {}  % Theorem head spec (can be left empty, meaning ‘normal’)

\theoremstyle{plain}
\newtheorem{lemma}{Lemma}[section]
\newtheorem{theorem}[lemma]{Theorem}
\newtheorem{fact}[lemma]{Fact}
\newtheorem{definition}[lemma]{Definition}
\newtheorem{corollary}[lemma]{Corollary}
\newtheorem{proposition}[lemma]{Proposition}

\theoremstyle{myremark}
\newtheorem{remark}[lemma]{Remark}

\newcommand{\nicec}{\mathcal{C}}

\newcommand{\er}{\mathbb{R}}
\newcommand{\rat}{\mathbb{Q}}

\newcommand{\zet}{\mathbb{Z}}

\newcommand{\htpyequiv}{\simeq}

\newcommand{\autom}{\mathrm{Aut}}

\renewcommand{\subset}{\subseteq}
\newcommand{\susp}{\Sigma}
\newcommand{\cl}{\mathrm{Cl}}

\newcommand{\lk}{\mathrm{lk}}
\newcommand{\st}{\mathrm{st}}

\newcommand{\overh}{\widetilde{H}}
\newcommand{\overk}{\widetilde{K}}
\newcommand{\overv}{\widetilde{V}}

\newcommand{\longversion}[1]{}
\begin{document}
\title{Transitivity is not a (big) restriction on homotopy types}
\author{Micha{\l} Adamaszek}
\address{Max Planck Institute for Informatics, 66123 Saarbr\"ucken, Germany}
\email{aszek@mimuw.edu.pl}
\keywords{Vertex-transitivity, Cayley graph, Clique complex, Golomb ruler}
\begin{abstract}
For every simplicial complex $K$ there exists a vertex-transitive simplicial complex homotopy equivalent to a wedge of copies of $K$ with some copies of the circle. It follows that every simplicial complex can occur as a homotopy wedge summand in some vertex-transitive complex. One can even demand that the vertex-transitive complex is the clique complex of a Cayley graph or that it is facet-transitive.
\end{abstract}

\maketitle

%%%%%%%%%%%%%%%%%%%%%%%%%%%%%%%%%%%%%%%%%%%%%%%%%%%%%%%%%%%
\section{Introduction}

This note is a result of considering the following problem: which homotopy types can be represented by simplicial complexes with vertex-transitive group actions?

Right from the outset we resolve the usual notational clash by declaring that the letter $G$ will denote graphs and $\Gamma$ --- groups. We will also use $K$ for simplicial complexes and $X$ for an object which can be a simplicial complex or a graph.

A simplicial complex or graph $X$ is \emph{vertex-transitive} if the automorphism group $\autom(X)$ acts transitively on the set of vertices $V(X)$. Objects with such rich symmetry groups appear naturally in various constructions in combinatorial topology and related areas. A good example are simplicial complexes encoding graph properties, see \cite{jakobsbook}, or Cayley graphs. Vertex-transitive triangulations of manifolds are particularly intriguing and a lot of effort went into finding minimal examples of those, see for instance \cite{KoLutz}. In this note we concentrate just on homotopy, asking to what extent vertex-transitive simplicial complexes of various kinds model homotopy types of finite CW-complexes.

\medskip
The following is a summary of the results we prove later in this note.
\begin{theorem}
\label{thm:super-main}
For any finite, connected simplicial complex $K$ there exists a vertex-transitive, finite simplicial complex $\widetilde{K}$ with a homotopy equivalence
\begin{equation}
\label{eq:first} \textstyle\widetilde{K} \htpyequiv \textstyle\bigvee^n K \vee \bigvee^l S^1
\end{equation}
for some $n\geq 1$, $l\geq 0$. Furthermore, one can require $\widetilde{K}$ to satisfy one of the following additional conditions:
\begin{itemize}
\item[(i)] $\autom(\widetilde{K})$ has a cyclic subgroup whose action is vertex-transitive,
\item[(ii)] $\widetilde{K}$ is the clique complex of a Cayley graph of some finite group,
\item[(iii)] $\widetilde{K}$ is simultaneously vertex-transitive and facet-transitive.
\end{itemize}
\end{theorem}
The relations between various parameters appearing in \eqref{eq:first} as well as the possible choices of the group with transitive action are discussed in the last section.

% Let us make a few comments. In (i) we emphasized the choice of $\zet/n$ because of its simplicity. We will see that also in (ii) and (iii) one can take $\Gamma$ to be a fairly explicit subgroup of $\Sigma_n\times Z_m$. In general, it will be clear that in all of (i), (ii) and (iii) one can use a plethora of groups, which are only required to satisfy a mild cancellation condition, respectively $\mathcal{G}(2)$, $\mathcal{G}(3)$ and $\mathcal{G}(4)$ of Definition~\ref{def:r}.

The construction which satisfies (i) is very simple, but it introduces a technique used later in the proofs of (ii) and (iii). We will also need some auxiliary results about clique and neighbourhood complexes. This seems to be the first proof that the clique complexes of Cayley graphs can have such diverse homotopy types. A simplicial complex $K$ is \emph{facet-transitive} if the induced action of $\autom(K)$ on the set of facets (maximal faces) of $K$ is transitive. For the relation between vertex- and facet-transitivity see \cite{frank}. The complexes constructed for (iii) will be instances of closed neighbourhood complexes of Cayley graphs.

\smallskip
The basic idea behind the constructions is to stick together multiple copies of $K$ in such a way that in the resulting complex $\widetilde{K}$ every vertex belongs to $|V(K)|$ copies of $K$ and it plays the role of a different vertex of $K$ in each copy. In particular, each vertex link $\lk_{\overk}(u)$ is isomorphic to the disjoint union $\bigsqcup_{v\in K}\lk_K(v)$. Such a complex is a special case of what we call a \emph{cluster} of copies of $K$. Similar ideas appear in \cite{FinkR,Fink}, but we will have to refine them significantly to satisfy the conditions of Theorem~\ref{thm:super-main}.

The paper is laid out as follows. After some preliminaries we give an elementary proof of (i) in Section~\ref{sect:simple} (Corollary~\ref{cor:cyclic}).  In Section~\ref{sect:clusters} we prove some auxiliary results about topological constructions on clusters of graphs. In Section \ref{sect:cliques} we extend the technique used for (i) to prove Proposition~\ref{cor:cliquesok}, which implies (ii) and (iii). Additional remarks are gathered in Section~\ref{sect:last}.

%%%%%%%%%%%%%%%%%%%%%%%%%%%%%%%%%%%%%%%%%%%%%%%%%%%%%%%%%%%
\section{Preliminaries}
\label{sect:prelim}

We assume the reader is familiar with basic concepts of combinatorial topology, such as simplicial complexes, acyclic matchings in the sense of Discrete Morse Theory, the nerve lemma and the notion of Cayley graphs of groups. A useful reference is \cite{kozbook}.

We refine the notion of transitivity by saying that a simplicial complex or graph $X$ is \emph{$\Gamma$-vertex-transitive} if there is a subgroup $\Gamma\subset\autom(X)$ acting transitively on the vertices of $X$. The same extends to $\Gamma$-facet-transitivity.

The \emph{open neighbourhood} of a vertex $v$ in a simple, loopless, undirected graph $G$ is $N_G(v)=\{w~:~vw\in E(G)\}$. The \emph{closed neighbourhood} is $N_G[v]=N_G(v)\cup\{v\}$. For a connected graph $G$ we define three simplicial complexes with vertex set $V(G)$.
\begin{itemize}
\item The \emph{clique complex} $\cl(G)$. Its faces are all the cliques (complete subgraphs) of $G$.
\item The \emph{(open) neighbourhood complex} $N(G)$. Its faces are all sets $\sigma$ such that $\sigma\subset N_G(v)$ for some $v\in V(G)$.
\item The \emph{closed neighbourhood complex} $N[G]$. Its faces are all sets $\sigma$ such that $\sigma\subset N_G[v]$ for some $v\in V(G)$.
\end{itemize}
The first two constructions are classical. Clique complexes have been studied in various branches of geometry and combinatorics as flag complexes, independence complexes or Vietoris-Rips complexes. The open neighbourhood complexes appeared in topological lower bounds for the chromatic number of $G$. Both $N(G)$ and $N[G]$ are special cases of \emph{distance-neighbourhood} complexes, defined recently in \cite{cory}.

Clearly all three constructions are $\Gamma$-vertex-transitive for a $\Gamma$-vertex-transitive graph $G$. Note that each maximal face of $N[G]$ is of the form $N_G[v]$ for some vertex $v\in V(G)$. It follows easily that if $G$ is $\Gamma$-vertex-transitive then $N[G]$ is not only $\Gamma$-vertex- but also $\Gamma$-facet-transitive. The same applies to $N(G)$.

\begin{definition}
\label{def:cluster}
Let $X$ be a finite simplicial complex (or graph) which is the union of its induced subcomplexes (or induced subgraphs) $$X=X_1\cup\cdots\cup X_k$$ such that $X_i\cap X_j$ is either empty or it consists of a single vertex for $i\neq j$. We then call $X$ a \emph{cluster} of $X_1,\ldots,X_k$ and we call $X_1,\ldots,X_k$ the \emph{parts} of $X$.
\end{definition}

For example, if $X_1\cap\cdots\cap X_k\neq\emptyset$ then the cluster is a one-point union of $X_1,\ldots,X_k$. If $X_i\cap X_j=\emptyset$ for all $i\neq j$ then we obtain the disjoint union of the parts $X_i$. Every connected component of a cluster is again a cluster. The cluster is not defined uniquely by its parts; for this one needs to know also their intersection pattern. Let us make this intuitive notion precise.

\begin{definition}
\label{def:intpattern}
The \emph{shape} of a cluster $X=X_1\cup\cdots\cup X_k$ is the hypergraph with vertex set $V(X)$ and with the set of hyperedges $\{V(X_i)~:~i=1,\ldots,k\}$.
\end{definition}

For simplicial complexes even the homotopy type of $X$ is not determined by those of $X_1,\ldots,X_k$. However, we clearly have the following fact.

\begin{fact}
\label{fact:union}
If a finite, connected simplicial complex $K$ is a cluster of connected subcomplexes $K_1,\ldots,K_k$ then there is a homotopy equivalence
$$\textstyle K\htpyequiv K_1\vee\cdots\vee K_k\vee\bigvee^l S^1$$
for some $l\geq 0$. Moreover, the value of $l$ depends only on the shape of the cluster.
\end{fact}
\begin{proof}
Consider the natural map $p:K_1\sqcup\cdots\sqcup K_k\to K$. For each vertex $x$ of $K$ the preimage $p^{-1}(x)$ is a finite collection of vertices which depends only on the shape of $K$. Let $K'$ be the space obtained from $K_1\sqcup\cdots\sqcup K_k$ by spanning a tree (one-dimensional contractible complex) on the vertices of each non-trivial preimage $p^{-1}(x)$, $x\in V(K)$. The space $K$ is obtained from $K'$ by contracting all those disjoint trees; that means $K\htpyequiv K'$. Since all the subspaces $K_i$ of $K'$, as well as $K'$ itself, are connected, a sequence of deformations shows that $K'$ is equivalent to a one-point union of all the $K_i$ and a graph. 
\end{proof}

The next definition captures the length of the shortest cycle formed by the parts of a cluster.

\begin{definition}
\label{def:helly}
Suppose $X$ is a cluster of $X_1,\ldots,X_k$. The \emph{cluster-girth} of $X$ is the smallest number $l\geq 3$ such that there is a cycle $v_0,X_{i_0},v_1,X_{i_1},\ldots,v_{l-1},X_{i_{l-1}},v_0$ with all $v_0,\ldots,v_{l-1}$ pairwise distinct, all $i_0,\ldots,i_{l-1}$ pairwise distinct and such that $v_j\in V(X_{i_{j-1}})$ and $v_j\in V(X_{i_j})$ for $j=0,\ldots,l-1$ (indices modulo $l$). The cluster-girth is infinite if there is no such cycle.
\end{definition}

We remark that this definition coincides with the standard definition of girth for hypergraphs \cite[Ch.5,\S 1]{berge}, i.e. the cluster-girth is equal to the hypergraph-theoretic girth of the cluster's shape. 

As an example, every graph $G$ can be viewed as a cluster of its own edges, treated as induced subgraphs $G[\{u,v\}]$ for $uv\in E(G)$. With this interpretation the cluster-girth of $G$ equals its girth as a graph. However, $G$ can also be treated as a cluster of induced subgraphs $G[\{u,v\}]$ over \emph{all} pairs $\{u,v\}\subseteq V(G)$. In this case the cluster-girth always equals $3$.

%%%%%%%%%%%%%%%%%%%%%%%%%%%%%%%%%%%%%%%%%%%%%%%%%%%%%%%%%%%
\section{A simple construction with cyclic symmetry}
\label{sect:simple}

We will now present a construction which proves the main statement of Theorem~\ref{thm:super-main} and has the property (i) mentioned in that theorem.

\medskip
Suppose $K$ is any connected simplicial complex with $d$ vertices $v_1,\ldots,v_d$. Pick $d$ integers $0\leq a_1<\cdots<a_d$ subject to the condition
\begin{equation}
\label{eq:golomb1} \mathrm{if}\ a_j-a_i=a_{j'}-a_{i'}\  \mathrm{then}\ i=j\ \mathrm{or}\ j=j'.
\end{equation}
It is easy to pick any such set of integers, for instance $a_i=2^i$ will do. The sets which satisfy condition \eqref{eq:golomb1} are called \emph{Golomb rulers} (with $d$ marks). If we care about optimality we can take one of the known Golomb rulers with $a_d\approx d^2$ (see \cite{ET}).

Next, pick an integer $n>2a_d$ coprime with all the differences $a_j-a_i$ for $j>i$, and define $\overk$ to be the simplicial complex with vertex set $V(\overk)=\zet/n$ whose faces are all sets of the form
\begin{equation}
\label{eq:face}
\{x+a_{i_1},\ldots,x+a_{i_k}\}
\end{equation}
where $x\in \zet/n$ and $\{v_{i_1},\ldots,v_{i_k}\}$ is a face of $K$ (addition modulo $n$).

The complex $\overk$ is clearly $\zet/n$-vertex-transitive and the coprimality requirement guarantees that it is connected (it is the only use of that condition). Another point of view on \eqref{eq:face} is as follows. The assignment $v_i\to a_i$ for $i=1,\ldots,d$ embeds $K$ into $\overk$ as a simplicial complex with vertex set $\{a_1,\ldots,a_d\}\subset\zet/n$. Then $\overk$ is the smallest extension of $K$ to a $\zet/n$-vertex-transitive complex. 

\bigskip
Part (i) of Theorem~\ref{thm:super-main} clearly follows from Fact~\ref{fact:union} once we prove the following proposition.

\begin{proposition}
\label{prop:cycliccomplex}
The complex $\overk$ defined above is a cluster of $n$ subcomplexes, each isomorphic to $K$.
\end{proposition}
\begin{proof}
For any $x\in\zet/n$ let $\overk_x$ be the subcomplex of $\overk$ induced by the vertex set $\overv_x=\{x+a_i~:~i=1,\ldots,d\}$. 

We first show that $\overk$ is a cluster of $\overk_0,\ldots,\overk_{n-1}$. The property $\overk=\bigcup_x\overk_x$ follows directly from the definition of $\overk$. Next, suppose that for some $x\neq y$ the intersection of $\overv_x$ and $\overv_y$ contains two vertices $p\neq q$. Then
$$p=x+a_i=y+a_{i'},\ q=x+a_j=y+a_{j'}\ \pmod{n}.$$
It follows that
$$(a_j-a_i)-(a_{j'}-a_{i'})=0.$$
The last equality holds a priori only modulo $n$, but since the left hand side belongs to the interval $[-2a_d,2a_d]$ and $n>2a_d$, it holds also in $\zet$.
By the Golomb property \eqref{eq:golomb1} we conclude that either $i=j$, which gives $p=q$, or $j=j'$, which gives $x=y$. In either case we have a contradiction, so the cluster claim is proved. 

It also follows that each face of $\overk$ of cardinality at least two belongs to exactly one $\overk_x$ and therefore for each $x\in\zet/n$ the obvious homomorphism of complexes $K\to\overk_x$ given by $v_i\to x+a_i$ is an isomorphism.
\end{proof}

\begin{corollary}
\label{cor:cyclic}
For any finite, connected simplicial complex $K$ there exists an integer $n$ and a $\zet/n$-transitive simplicial complex $\overk$ with $n$ vertices which satisfies a homotopy equivalence
$$\overk\htpyequiv \textstyle \bigvee^n K\vee \bigvee^l S^1$$
for some $l\geq 0$.
\end{corollary}
%%%%%%%%%%%%%%%%%%%%%%%%%%%%%%%%%%%%%%%%%%%%%%%%%%%%%%%%%%%
\section{Clusters and topological constructions}
\label{sect:clusters}

In this section we prepare some prerequisites for the proof of parts (ii) and (iii) of Theorem~\ref{thm:super-main}. 

\begin{proposition}
\label{prop:clusters-topol}
Suppose a graph $G$ is a cluster of induced subgraphs $G_1,\ldots,G_k$.
\begin{itemize}
\item[(i)] If $G$ has cluster-girth at least $4$ then $\cl(G)$ is isomorphic to a cluster of complexes\\ $\cl(G_1),\ldots,\cl(G_k)$ of the same shape as $G$.
\item[(ii)] If $G$ has cluster-girth at least $5$ then $N[G]$ is homotopy equivalent to a cluster of complexes $N[G_1],\ldots,N[G_k]$ of the same shape as $G$.
\end{itemize}
\end{proposition}

Before proceeding with the proof it is helpful to note the following easy consequence of the definition of a cluster. Suppose that $G$ is a cluster of graphs with parts $G_1,\ldots,G_k$ and $xy\in E(G)$. Then there is a \emph{unique} index $i$ such that $xy\in E(G_i)$.

\begin{proof}
Part (i) is obvious since cluster-girth at least $4$ guarantees that every triangle of $G$ is contained in one of the $G_i$.

\smallskip
We proceed to prove (ii). Clearly $N[G]$ contains a cluster of $N[G_1],\ldots,N[G_k]$ as a subcomplex. We will show that $N[G]$ deformation retracts to this subcomplex by collapsing away all the faces which span across more than one vertex set $V(G_i)$. For this we need to verify two claims. 

\begin{itemize}
\item[(*)] Suppose $\sigma$ is a face of $N[G]$ which is not contained in any $V(G_i)$. Then there is \emph{exactly one} vertex $v\in V(G)$ for which $\sigma\subseteq N_G[v]$. 
\item[(**)] If $\sigma\subseteq V(G_i)$ is a face of $N[G]$ of cardinality at least $2$, then for any vertex $v\in V(G)$ such that $\sigma\subseteq N_G[v]$ we have $v\in V(G_i)$.
\end{itemize}

\begin{proof}[Proof of (*)] Such a vertex $v$ exists by the definition of faces of $N[G]$. Now suppose that there are $v_1\neq v_2$ such that $\sigma\subseteq N_G[v_1]\cap N_G[v_2]$. 

First suppose that $v_1,v_2\not\in\sigma$. Pick $u_1,u_2\in\sigma$ such that $v_1u_1\in E(G_i)$ and $v_1u_2\in E(G_j)$ with $i\neq j$. Such a choice is possible, as otherwise we would have $\sigma\subseteq V(G_i)$ for some $i$. Next, define $i',j'$ by $v_2u_1\in E(G_{i'})$ and $v_2u_2\in E(G_{j'})$. If we had $u_1\in V(G_j)$ then $v_1,u_2\in V(G_i)\cap V(G_j)$, so we have $u_1\not\in V(G_j)$ and, by symmetry, $u_2\not\in V(G_i)$. It follows that $i'\neq j$ and $i\neq j'$.

If $i'=i$ and $j'=j$ then $V(G_i)\cap V(G_j)$ contains $v_1$ and $v_2$, a contradiction.

If $i'=i$ and $j'\neq j$ then
$$v_1,G_i,v_2,G_{j'},u_2,G_j,v_1$$
is a cluster-cycle of length $3$. It follows that we can assume $i\neq i'$ and, by symmetry, $j\neq j'$. Now if $i'=j'$ then we have a cycle
$$v_1,G_i,u_1,G_{i'},u_2,G_j,v_1$$
hence $i'\neq j'$. But then there is a cycle of clusters
$$v_1,G_i,u_1,G_{i'},v_2,G_{j'},u_2,G_j,v_1$$
of length $4$. That gives a contradiction in the case $v_1,v_2\not\in\sigma$.

Now suppose $v_1\in\sigma$. Define $i$ by the condition $v_1v_2\in E(G_i)$. Let $u\in\sigma$ be any vertex such that $u\not\in V(G_i)$. Then $u\neq v_1,v_2$. Now $v_1,u,v_2$ is a triangle in $G$ whose one edge is in $E(G_i)$ and its opposite vertex is not in $V(G_i)$. That easily implies that the cluster-girth of $G$ is $3$, and the proof of (*) is complete.
\end{proof}

\begin{proof}[Proof of (**)]
Suppose that $v\not\in V(G_i)$. Let $v_1\neq v_2$ be any two vertices of $\sigma$. We have $v\neq v_1$ and $v\neq v_2$ since $v_1,v_2\in V(G_i)$ and $v\not\in V(G_i)$. If $j,l$ are defined by $vv_1\in E(G_j)$ and $vv_2\in E(G_l)$ then $j\neq i$ and $l\neq i$ since otherwise $v\in V(G_i)$. Also $j\neq l$ as otherwise $V(G_i)\cap V(G_j)$ would contain $v_1$ and $v_2$. Now $v,G_j,v_1,G_i,v_2,G_l,v$ is a cycle which shows that the cluster-girth of $G$ is $3$. That contradiction proves (**).
\end{proof}

\smallskip
We can now continue the proof of the proposition. Let $A\subset N[G]$ be the subcomplex consisting of all those faces $\sigma$ which are contained in some $V(G_i)$, formally $A=\bigcup_{i=1}^k N[G]\big[ V(G_i) \big]$, and let $B=N[G]\setminus A$. By (**) we have $N[G]\big[ V(G_i) \big]=N[G_i]$, which means that $A$ is a cluster of $N[G_i]$, $i=1,\ldots,k$.

By (*) $B$ can be written as a disjoint union $B=\bigsqcup_{v\in V(G)} B_v$ where 
$$B_v=\{\sigma\in B~:~\sigma\subseteq N_G[v]\}.$$
Consider the matching $\mathcal{M}_v$ on $B_v$ defined as
$$\mathcal{M}_v=\{(\sigma,\sigma\cup\{v\})~:~\sigma\in B_v,\ v\not\in\sigma\}.$$
Clearly $\mathcal{M}_v$ is acyclic. The matching $\mathcal{M}=\bigcup_v\mathcal{M}_v$ on $N[G]$ is acyclic by the Patchwork Theorem \cite[Thm. 11.10]{kozbook} since for $\sigma\in B_v$, $\tau\in B_w$ and $v\neq w$ we have $\sigma\not\subset\tau$ by (*).

The set of critical faces of $\mathcal{M}$ forms the subcomplex $A\subset N[G]$, therefore $N[G]$ simplicially collapses to $A$ and (ii) is proved.
\end{proof}

\begin{remark}
The matching $\mathcal{M}_v$ used to collapse $B_v$ in the proof above boils down to the following observation. Let $\Delta(V)$ denote the simplex with vertex set $V$. If $V=\{v\}\sqcup V_1\sqcup\cdots\sqcup V_l$ then $\Delta(V)$ collapses onto the subcomplex $\Delta(\{v\}\cup V_1)\vee\cdots\vee\Delta(\{v\}\cup V_l)$. 
\end{remark}

\begin{figure}
\label{fig:5gon}
\includegraphics[scale=0.6]{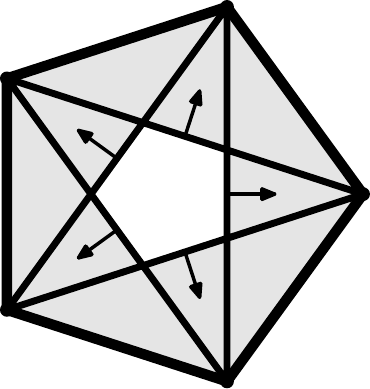}
\caption{The $5$-cycle graph $G$ (outer cycle) viewed as a cluster of its edges and the complex $N[G]$. The arrows indicate the matching from Proposition~\ref{prop:clusters-topol}.(ii). Cluster-girth $5$ ensures that for each face of $N[G]$ which spans across more than one part of the cluster there is a unique choice of the collapsing direction.}
\end{figure}

In order to construct facet-transitive complexes with prescribed homotopy types we will use the following universality property of closed neighbourhood complexes.

\begin{lemma}
\label{lem:neibhtpy}
For any finite, connected simplicial complex $K$ there exists a graph $G$ with a homotopy equivalence $N[G]\htpyequiv K$. In fact, one can take $G=(\mathrm{bd}\ K)^{(1)}$, the $1$-skeleton of the barycentric subdivision of $K$.
\end{lemma}

This was essentially shown by Csorba \cite{csorba}, although his argument is adapted to the open neighbourhood complexes $N(G)$. We include the proof for completeness.

\begin{proof}
Let $G=(\mathrm{bd}\ K)^{(1)}$ be as above. Our notation will identify the vertices of $\cl(G)$ and the faces of $K$. Since $K$ is homeomorphic to $\cl(G)$ it suffices to show that $\cl(G)\htpyequiv N[G]$. Consider the following cover $\nicec$
$$\cl(G)=\bigcup_{\emptyset\neq \sigma\in K}\st_{\cl(G)}(\sigma)$$
of $\cl(G)$ by the collection of all vertex stars. The intersection $\st_{\cl(G)}(\sigma_1)\cap\cdots\cap \st_{\cl(G)}(\sigma_k)$ contains a vertex $\tau$ if and only if $\tau\in N_G[\sigma_1]\cap\cdots\cap N_G[\sigma_k]$, or equivalently $\sigma_1,\ldots,\sigma_k\in N_G[\tau]$. That gives an isomorphism between the nerve $\mathrm{Nrv}(\nicec)$ and $N[G]$. The proof will be completed if we can show that each nonempty intersection of the form $\st_{\cl(G)}(\sigma_1)\cap\cdots\cap \st_{\cl(G)}(\sigma_k)$ is contractible, since then the nerve lemma yields $\cl(G)\htpyequiv \mathrm{Nrv}(\nicec)\equiv N[G]$.

As $\cl(G)$ is a flag complex it suffices to show that each nonempty set of the form $V=N_G[\sigma_1]\cap\cdots\cap N_G[\sigma_k]$ induces a cone in $G$ (A cone is a graph which contains a vertex adjacent to all other vertices; the clique complex of a cone is contractible). For this suppose $\tau\in V$, that is $\sigma_1,\ldots,\sigma_k\in N_G[\tau]$ for some simplex $\emptyset\neq\tau\in K$. Set $S=\{i~:~\tau\subseteq\sigma_i\}$. For $j\not\in S$ we have $\sigma_j\subseteq\tau$. We claim that:
\begin{itemize}
\item[a)] if $S=\emptyset$ then $V$ induces a cone with apex $\widetilde{\sigma}=\bigcup_{i=1}^k\sigma_i$.
\item[b)] if $S\neq\emptyset$ then $V$ induces a cone with apex $\widetilde{\sigma}=\bigcap_{i\in S} \sigma_i$.
\end{itemize}
We need to check that $\widetilde{\sigma}\in V$ and that $\widetilde{\sigma}$ is adjacent to all vertices of $V$. For this, let $\tau'$ be any element of $V$, that is assume $\sigma_1,\ldots,\sigma_k\in N_G[\tau']$. 

In a), $\widetilde{\sigma}$ is a face of $K$ since $\widetilde{\sigma}\subseteq\tau$. We have $\widetilde{\sigma}\in V$ since $\sigma_i\subset\widetilde{\sigma}$ for all $i$. Now if $\tau'\subseteq\sigma_i$ for some $i$ then $\tau'\subseteq\widetilde{\sigma}$. Otherwise $\sigma_i\subseteq\tau'$ for all $i$ and then $\widetilde{\sigma}\subseteq\tau'$.

In b), $\widetilde{\sigma}$ is nonempty since it contains $\tau$. We have $\widetilde{\sigma}\in V$ since $\sigma_j\subset\tau\subseteq\widetilde{\sigma}$ for all $j\not\in S$ and $\widetilde{\sigma}\subseteq\sigma_i$ for all $i\in S$. Now if $\sigma_i\subseteq\tau'$ for some $i\in S$ then $\widetilde{\sigma}\subseteq\tau'$. Otherwise $\tau'\subseteq\sigma_i$ for all $i\in S$ and then $\tau'\subseteq\widetilde{\sigma}$.

\end{proof}

%%%%%%%%%%%%%%%%%%%%%%%%%%%%%%%%%%%%%%%%%%%%%%%%%%%%%%%%%%%
\section{Complexes arising from Cayley graphs}
\label{sect:cliques}
This section culminates with Proposition~\ref{cor:cliquesok}, which is a recipe for clustering copies of the same complex into a clique complex or neighbourhood complex (up to homotopy) of some Cayley graph. To construct that Cayley graph we first need a group. We introduce the following notational convention: a pair $(\Gamma,D)$ consists of a finite group $\Gamma$, with neutral element $e$, together with a distinguished set $D=\{\gamma_1,\ldots,\gamma_d\}\subseteq\Gamma$ (the elements in $D$ are not required to generate $\Gamma$). The cardinality of $D$ is denoted by $|D|=d$. Typically $d$ will be small compared to $|\Gamma|$.

\smallskip
By comparison with Section~\ref{sect:simple} the object we now define could be referred to as a ``generalized Golomb ruler in the group $\Gamma$''.
\begin{definition}
\label{def:r}
We say a pair $(\Gamma,D)$ as above satisfies condition $\mathcal{R}(p)$, for $p\geq 1$, if for every sequence of indices $i_0,\ldots,i_{2p-1}\in\{1,\ldots,d\}$ the equality
$$\gamma_{i_0}{\gamma_{i_1}}^{-1}\gamma_{i_2}{\gamma_{i_3}}^{-1}\cdots\gamma_{i_{2p-2}}{\gamma_{i_{2p-1}}}^{-1}=e$$
implies $i_l=i_{l+1}$ for some $l=0,\ldots,2p-1$ (indices modulo $2p$). We further say $(\Gamma,D)$ satisfies condition $\mathcal{G}(p)$ if it satisfies $\mathcal{R}(1),\ldots,\mathcal{R}(p)$.
\end{definition}

We will now show how to glue together copies of a given graph $H$ using a suitable pair $(\Gamma,D)$.
\begin{definition}
\label{def:htilde}
Given a connected graph $H$ on $d$ nodes $v_1,\ldots,v_d$ and a pair $(\Gamma,D)$ with $|D|=d$ we construct a $\Gamma$-vertex-transitive graph $\overh_{\Gamma,D}$ as follows. The vertex set of $\overh_{\Gamma,D}$ is $V(\overh_{\Gamma,D})=\Gamma$. The edge set is
$$E(\overh_{\Gamma,D})=\big\{(\gamma\gamma_i,\gamma\gamma_j)~:~\gamma\in\Gamma,\ v_iv_j\in E(H)\big\}.$$
\end{definition}

\begin{proposition}
\label{prop:overhiscluster}
Suppose $H$ is a connected graph with $d$ vertices and $(\Gamma,D)$ is a pair with $|D|=d$.
\begin{itemize}
\item[a)] If $(\Gamma,D)$ satisfies condition $\mathcal{G}(2)$ then the graph $\overh_{\Gamma,D}$ is a cluster of $|\Gamma|$ copies of $H$. 
\item[b)] If $(\Gamma,D)$ satisfies $\mathcal{G}(p)$ then the cluster-girth of $\overh_{\Gamma,D}$ is at least $p+1$.
\item[c)] The connected components of $\overh_{\Gamma,D}$ are isomorphic. The component of the identity is a Cayley graph of some subgroup of $\Gamma$.
\end{itemize}
\end{proposition}
\begin{proof}
For simplicity we denote $\overh:=\overh_{\Gamma,D}$. For any $\gamma\in \Gamma$ let $\overh_\gamma$ be the subgraph of $\overh$ induced by the vertex set $\overv_\gamma=\{\gamma\gamma_i~:~i=1,\ldots,d\}$. 

Part a). We show that $\overh$ is a cluster of $\overh_\gamma$ for $\gamma\in\Gamma$. Clearly $\overh=\bigcup_\gamma\overh_\gamma$. Suppose the intersection of $\overv_x$ and $\overv_y$ for $x\neq y$ contains two vertices $p\neq q$. Then
$$p=x\gamma_i=y\gamma_{i'},\ q=y\gamma_j=x\gamma_{j'}$$
for some $i,i',j,j'$. It follows that
$$\gamma_i{\gamma_{i'}}^{-1}\gamma_{j}{\gamma_{j'}}^{-1}=x^{-1}yy^{-1}x=e.$$
By $\mathcal{R}(2)$ we conclude that, without loss of generality, either $i=i'$, which gives $x=y$, or $i=j'$, which gives $p=q$. In either case we have a contradiction, so the cluster claim is proved. 

This also implies that each graph $\overh_\gamma$ is isomorphic to $H$ and the proof of a) is finished.

\smallskip
Part b). Suppose $v_0,\overh_{\tau_0},v_1,\ldots,v_{k-1},\overh_{\tau_{k-1}},v_0$ is a cycle such that $v_l\in V(\overh_{\tau_{l-1}})\cap  V(\overh_{\tau_{l}})$ for $l=0,\ldots,k-1$, with all $v_0,\ldots,v_{k-1}\in \Gamma$ pairwise distinct and $\tau_0,\ldots,\tau_{k-1}\in\Gamma$ pairwise distinct. There exists elements $\gamma_{i_l},\gamma_{i_l'}\in D$, $l=0,\ldots,k-1$, for which
$$v_l=\tau_{l-1}\gamma_{i_l}=\tau_l\gamma_{i_l'},\ \ l=0,\ldots,k-1.$$
Then
$$\gamma_{i_0}{\gamma_{i_0'}}^{-1}\cdots\gamma_{i_{k-1}}{\gamma_{i_{k-1}'}}^{-1}=\tau_{k-1}^{-1}\tau_0\tau_0^{-1}\tau_1\cdots\tau_{k-2}^{-1}\tau_{k-1}=e.$$
If $k\leq p$ then $\mathcal{R}(k)$ yields, without loss of generality, either $i_0=i_0'$, but then $\tau_0=\tau_{k-1}$ or $i_0=i_{k-1}'$, and then $v_0=v_{k-1}$. In either case we have a contradiction which shows that $k\geq p+1$.

\smallskip
Part c). By Definition~\ref{def:htilde}, there is an edge from $x$ to $y$ in $\overh$ if and only if $x^{-1}y\in S$, where
$$S=\{ \gamma_i^{-1}\gamma_j~:~i,j=1,\ldots,d,\ v_iv_j\in E(H)\}.$$
It means that $\overh$ is the (possibly disconnected) Cayley graph of $\Gamma$ with respect to the (possibly non-generating) set $S$. The connected component of identity in $\overh$ is the Cayley graph of the subgroup $\Gamma'=\langle S\rangle\subseteq\Gamma$ with respect to $S$.
\end{proof}

The existence of pairs $(\Gamma,D)$ which satisfy conditions $\mathcal{G}(p)$ can be easily deduced from the relatively well-known fact that there exist finite groups whose Cayley graphs have arbitrarily high degree and girth. The details are given in the next proposition.

\begin{proposition}
\label{prop:groupsexist}
For every $p\geq 2$, $d\geq 3$ there exists a pair $(\Gamma,D)$ with $|D|=d$ which satisfies condition $\mathcal{G}(p)$.
\end{proposition}
\begin{proof}
By a result of Biggs \cite[Sect. 4]{biggs} (see also \cite[Theorem 19]{survey}) for any $d\geq 3$ and any $p$ there exist a finite group $\Gamma=\langle\gamma_1,\ldots,\gamma_d\rangle$, generated by $d$ involutions $\gamma_1,\ldots,\gamma_d$ and such that any cyclically reduced word of length at most $2p$ in the generators $\gamma_i$ gives a nontrivial element of $\Gamma$. Since in this case $\gamma_i^{-1}=\gamma_i$, this implies that the pair $(\Gamma, \{\gamma_1,\ldots,\gamma_d\})$ satisfies $\mathcal{G}(p)$.
\end{proof}

It is now a matter of putting all the pieces together to derive the main result.

\begin{proposition}
\label{cor:cliquesok}
For any finite, connected simplicial complex $K$ there exists an integer $n\geq 1$ and an $n$-vertex Cayley graph $G$ which satisfies homotopy equivalences
\begin{equation*}
\textstyle N[G]\htpyequiv\cl(G)\htpyequiv \bigvee^n K\vee\bigvee^l S^1
\end{equation*}
for some $l\geq 0$.
\end{proposition}
As explained before that implies Theorem~\ref{thm:super-main}: part (ii) follows since $G$ is a Cayley graph and part (iii) since $N[G]$ is both vertex- and facet-transitive.
\begin{proof}
Let $H=(\mathrm{bd} K)^{(1)}$ be the $1$-skeleton of the barycentric subdivision of $K$. Choose any pair $(\Gamma,D)$ which satisfies $\mathcal{G}(4)$ and such that $|D|=|V(H)|$. This is possible by Proposition~\ref{prop:groupsexist}. Let $G$ be the connected component of $\overh_{\Gamma,D}$ containing the identity\footnote{In Section~\ref{sect:simple} we had an additional condition which guaranteed connectedness of the construction, but repeating that in the context of an arbitrary group $\Gamma$ would only obscure the proof. It is easier to consider the component of identity instead.}. By Proposition~\ref{prop:overhiscluster} the graph $G$ is a Cayley graph of some subgroup $\Gamma'\subseteq \Gamma$, it has $\frac{|\Gamma|}{[\Gamma:\Gamma']}=|\Gamma'|$ vertices and it is a cluster of $|\Gamma'|$ copies of $H$ with cluster-girth at least $5$. By Proposition~\ref{prop:clusters-topol} we have that 
\begin{itemize}
\item[a)] $\cl(G)$ is a cluster of copies of $\cl(H)$,
\item[b)] $N[G]$ is homotopy equivalent to a cluster of copies of $N[H]$.
\end{itemize}
From Fact~\ref{fact:union} and Lemma~\ref{lem:neibhtpy} we get:
\begin{align*}
\textstyle\cl(G)&\htpyequiv\textstyle \bigvee^n\cl(H)\vee\bigvee^l S^1\htpyequiv \bigvee^n K\vee\bigvee^l S^1\\
N[G]& \htpyequiv\textstyle \bigvee^nN[H]\vee\bigvee^lS^1\htpyequiv \bigvee^n K\vee\bigvee^l S^1,
\end{align*}
where $n=|V(G)|=|\Gamma'|$ and $l\geq 0$. The values of $n,l$ are the same for both decompositions since the clusters mentioned in a) and b) both have the same shape as $G$.
\end{proof}

As we see from the proof, in order to obtain Theorem~\ref{thm:super-main} one does not require the high girth construction of Proposition~\ref{prop:groupsexist} in its full generality. We only need pairs $(\Gamma,D)$ which satisfy condition $\mathcal{G}(4)$ (or just $\mathcal{G}(3)$ to prove Theorem~\ref{thm:super-main}.(ii)). We provide such pairs in the next proposition. This will make the argument self-contained and also demonstrate that the group whose action is transitive can be chosen to be fairly simple and explicit.

We say a sequence $a_1<\ldots<a_d$ of integers is \emph{(arithmetic) progression-free} if $a_i+a_j=2a_k$ for $1\leq i,j,k\leq d$ implies $i=j=k$. 

\begin{proposition}
\label{prop:groupeasy}
For any $i=2,\ldots,d+1$ let $\sigma_i\in \Sigma_{d+1}$ be the transposition $(1\ i)$. Let $0<a_1<\ldots<a_d$ be a progression-free sequence and choose any $m_d>4a_d$.

Consider the group $$\Gamma_d=\Sigma_{d+1}\oplus \zet/m_d$$ together with elements
$$\gamma_i=(\sigma_{i+1},\ a_i),\ i=1,\ldots,d.$$
Then $(\Gamma_d,\{\gamma_1,\ldots,\gamma_d\})$ satisfies $\mathcal{G}(4)$.
\end{proposition}
As will be clear from the proof, there are many gadgets one could use in place of the factor $\zet/m_d$ in order to achieve the same result.
\begin{proof}
We first make the following assertions about compositions of transpositions in $\Sigma_{d+1}$.
\begin{eqnarray}
\label{eq:sigma1}\sigma_{i_0}\sigma_{i_1}\sigma_{i_2}\sigma_{i_3}=\textrm{id}&\Longrightarrow&i_l=i_{l+1}\ \textrm{for some}\ l=0,\ldots,3. \\
\label{eq:sigma2}\sigma_{i_0}\sigma_{i_1}\sigma_{i_2}\sigma_{i_3}\sigma_{i_4}\sigma_{i_5}=\textrm{id}&\Longrightarrow&i_l=i_{l+1}\ \textrm{for some}\ l=0,\ldots,5\ \mathrm{or}\\
\nonumber & &\quad (i_0,\ldots,i_5)=(i,j,i,j,i,j).\\
\label{eq:sigma3}\sigma_{i_0}\sigma_{i_1}\sigma_{i_2}\sigma_{i_3}\sigma_{i_4}\sigma_{i_5}\sigma_{i_6}\sigma_{i_7}=\textrm{id}&\Longrightarrow&i_l=i_{l+1}\ \textrm{for some}\ l=0,\ldots,7\ \mathrm{or}\\
\nonumber & &\quad (i_0,\ldots,i_7)\ \textrm{is one of the sequences}\\
\nonumber & &\qquad (i,j,i,k,i,j,i,k),\\
\nonumber & &\qquad (i,j,i,k,j,i,j,k),\\
\nonumber & &\quad \textrm{or their cyclic shifts.}
\end{eqnarray}
This can be checked directly, preferably (as the author did) with the aid of a short computer program.

Now we can prove the proposition. If $\gamma_{i_0}{\gamma_{i_1}}^{-1}\gamma_{i_2}{\gamma_{i_3}}^{-1}=e$ then $\sigma_{i_0}\sigma_{i_1}\sigma_{i_2}\sigma_{i_3}=\mathrm{id}$ and by \eqref{eq:sigma1} we conclude that $\mathcal{R}(2)$ holds.

Next, suppose that $\gamma_{i_0}{\gamma_{i_1}}^{-1}\gamma_{i_2}{\gamma_{i_3}}^{-1}\gamma_{i_4}{\gamma_{i_5}}^{-1}=e$. By \eqref{eq:sigma2} either $i_l=i_{l+1}$ for some $l=0,\ldots,5$, and then $\mathcal{R}(3)$ holds, or $(i_0,\dots,i_5)=(i,j,i,j,i,j)$. In the latter case we have $3(a_i-a_j)\equiv 0\pmod{m_d}$. Since $-m_d<3(a_i-a_j)<m_d$ that implies $i=j$ and we also proved $\mathcal{R}(3)$.

Finally, suppose that $\gamma_{i_0}{\gamma_{i_1}}^{-1}\gamma_{i_2}{\gamma_{i_3}}^{-1}\gamma_{i_4}{\gamma_{i_5}}^{-1}\gamma_{i_6}{\gamma_{i_7}}^{-1}=e$. If $i_l=i_{l+1}$ for some $l=0,\ldots,7$ then we have $\mathcal{R}(4)$. Otherwise, rotating the indices if necessary, we can assume that $(i_0,\dots,i_7)$ is one of the two sequences appearing in \eqref{eq:sigma3}. That yields, respectively
\begin{eqnarray*}
4a_i-2a_j-2a_k&\equiv& 0\pmod{m_d},\ \textrm{or}\\
a_i+a_j-2a_k&\equiv& 0\pmod{m_d}.
\end{eqnarray*}
The values of the left-hand sides are in $(-m_d,m_d)$, hence the above equations hold in fact over $\zet$. That implies $i=j=k$. by our choice of the sequence $a_1,\ldots,a_d$. Now $\mathcal{R}(4)$ is proved.

Since $\mathcal{R}(2)$, $\mathcal{R}(3)$, $\mathcal{R}(4)$ hold, we showed $\mathcal{G}(4)$ as required.
\end{proof}

%%%%%%%%%%%%%%%%%%%%%%%%%%%%%%%%%%%%%%%%%%%%%%%%%%%%%%%%%%%
\section{Further results and conclusions}
\label{sect:last}

In this section we gather some comments and ramifications of our technique.

\subsection*{Choice of parameters} Our method always produces complexes $\overk$ such that in \eqref{eq:first} we have $n=|V(\overk)|=|\Gamma|$ where $\Gamma$ is the group which acts transitively. Let $d=|V(K)|$ and $d'=|V(\mathrm{bd}\ K)|$. By inspecting the proofs we easily find that:
\begin{itemize}
\item If the construction in (i) is carried out using the optimal Golomb rulers then $n\approx d^2$.
\item If the constructions in (ii) and (iii) are carried out using the pairs from Proposition~\ref{prop:groupeasy} then $n\approx (d')!$, which can be doubly-exponential in $d$. An application of Proposition~\ref{prop:groupsexist} also gives a doubly-exponential, but worse, dependence on $d$.
\end{itemize}
We also see that a variety of groups $\Gamma$ can be used. Theorem~\ref{thm:super-main}.(i) can be shown using any pair $(\Gamma,D)$ with $|D|=d$ satisfying $\mathcal{G}(2)$. To prove (ii) and (iii) one need pairs $(\Gamma, D')$, $|D'|=d'$, satisfying, respectively, $\mathcal{G}(3)$ or $\mathcal{G}(4)$. In each case the resulting complex will be transitive with respect to some subgroup of $\Gamma$.

\smallskip
Finally, note that every pair $(\Gamma,D)$ satisfies condition $\mathcal{G}(1)$. Golomb rulers from Section~\ref{sect:simple} can be alternatively defined as pairs $(\zet,D)$ which satisfy $\mathcal{G}(2)$. However, in any pair $(\Gamma,D)$ which satisfies $\mathcal{G}(3)$, the group $\Gamma$ must be non-abelian. The reason is the identity
$$e=\gamma_1{\gamma_2}^{-1}\gamma_3{\gamma_1}^{-1}\gamma_2{\gamma_3}^{-1}$$
which holds in any abelian group and refutes condition $\mathcal{R}(3)$, regardless of the choice of $D$. It would be interesting to see if parts (ii) and (iii) of Theorem~\ref{thm:super-main} hold if one requires a transitive action by an \emph{abelian} group.

\subsection*{Application to graph decompositions} A graph $H$ is called an \emph{induced isopart} in $G$ if $G$ is a union of induced subgraphs, each isomorphic to $H$. Fink and Ruiz \cite{FinkR} prove that every finite connected graph is an induced isopart of a circulant (i.e. a Cayley graph of a cyclic group). Fink \cite[Thm.2]{Fink} extends this to Cayley graphs of some other groups. Our construction resembles theirs but yields a slightly stronger result, namely a cluster of copies of $H$, and makes less restrictive assumptions regarding the choice of the group.

\subsection*{Modelling homotopy types}
It is an intriguing question whether all connected homotopy types can be modelled by vertex-transitive (or even more symmetric) simplicial complexes, and also how the answer depends on the group $\Gamma$ which acts transitively. Very little is known in this respect. The next proposition gives a non-trivial restriction in case of actions of cyclic groups.

\begin{proposition}
\label{prop:lefschetz}
If $K$ is a non-contractible, finite simplicial complex such that
\begin{itemize}
\item[i)] for every $i$ we have $\dim H_i(K;\rat)\leq 1$,
\item[ii)] the Euler characteristic $\chi(K)=\sum_i(-1)^i\dim H_i(K;\rat)$ is odd,
\end{itemize}
then $K$ is not homotopy equivalent to a $\zet/n$-vertex-transitive simplicial complex for any $n$.
\end{proposition}
For example, neither $\er P^2$ nor $S^2\vee S^1$ are homotopy equivalent to a $\zet/n$-vertex-transitive complex for any $n$.
\begin{proof}
Suppose $K$ is a $\zet/n$-transitive simplicial complex which satisfies i) and ii). We will show that $K$ is a simplex. Let $f:K\to K$ be the action of the generator of $\zet/n$. The map $f$ induces on each non-zero group $H_i(K;\rat)$ multiplication by $\pm 1$ and therefore the Lefschetz number of $f$ satisfies
\begin{eqnarray*}
\Lambda(f)&=&\sum_i(-1)^i\mathrm{tr}(f_*:H_i(K;\rat)\to H_i(K;\rat)) =\\
&=&\sum_i(-1)^i\cdot(\pm 1)\cdot\dim H_i(K;\rat)\equiv \chi(K)\pmod{2}.
\end{eqnarray*}
By ii) we get $\Lambda(f)\neq 0$, and so $f$ has a fixed point by the Lefschetz theorem. It follows that $K$ has a $\zet/n$-fixed point. Combined with vertex-transitivity of the $\zet/n$-action that means $K$ is a simplex.
\end{proof}

\subsection*{Connectivity} By using joins we can increase the connectivity of the construction in Theorem~\ref{thm:super-main}. If $K$ is connected then the $t$-fold join $\overk^{\ast t}=\underbrace{\overk\ast\cdots\ast\overk}_t$ of $\overk$ with itself is a $(2t-2)$-connected, vertex-transitive simplicial complex which contains the suspension $\susp^{2t-2}K$ as a homotopy direct summand. The complex $\overk^{\ast t}$ has vertex-transitive actions of groups such as $\Gamma\times\zet/t\subseteq \Gamma\times \Sigma_t\subseteq \Gamma\wr \Sigma_t$, where $\Gamma$ was the group acting on $\overk$.

\subsection*{Examples}
We end with some inspirational examples of torsion in  $\zet/n$-vertex- and simultaneously $\zet/n$-facet-transitive complexes found by computer search. 
\begin{itemize}
\item Let $C_n^r$ be the $r$-th power of the $n$-cycle, i.e. the graph obtained from the $n$-cycle $C_n$ by connecting each pair of vertices in distance at most $r$. Then for $k\geq 6$ the open neighbourhood complex $N(C_{2k+1}^3)$ has homology groups $$(H_0,H_1,H_2,\ldots)=(\zet,\zet,\zet/2,0,\ldots).$$
\item Let $K_{4k+2}$ be the simplicial complex with vertex set $\zet/(4k+2)$ and maximal faces $\{x,x+1,x+2,x+4,x+2k+4\}$ for $x\in\zet/(4k+2)$ (addition modulo $4k+2$). Then for $k\geq 2$ the homology of $K_{4k+2}$ is $$(H_0,H_1,H_2,\ldots)=(\zet,\zet\oplus\zet/2,0,\ldots).$$
\end{itemize}

%%%%%%%%%%%%%%%%%%%%%%%%%%%%%%%%%%%%%%%%%%%%%%%%%%%%%%%%%%%%%%%%%%%%%%
\subsection*{Acknowledgements.} I would like to thank Henry Adams, Florian Frick, Frank Lutz, Christopher Peterson and Corrine Previte for inspiration, discussions and references.

%%%%%%%%%%%%%%%%%%%%%%%%%%%%%%%%%%%%%%%%%%%%%%%%%%%%%%%%%%%

\end{document}